\documentclass[11pt]{article}
\usepackage{amsmath} 
\usepackage{verbatim}
\usepackage{amssymb, float}
\usepackage{color}
\usepackage{graphicx,epsfig}
\graphicspath{{figures/}}
\usepackage[applemac]{inputenc}
\usepackage{bbm}
\usepackage{amsmath} 
\usepackage{verbatim}
\usepackage{amssymb, float,empheq}
\usepackage{color}
\usepackage{graphicx,epsfig}
\graphicspath{{figures/}}
\usepackage[applemac]{inputenc}
\usepackage{graphicx}

\title{{Well-posedness and nonsmooth Lyapunov pairs for  state-dependent maximal monotone differential inclusions}}

\def\beq{\begin{equation}}
\def\eeq{\end{equation}}
\def\baq{\begin{eqnarray}}
\def\eaq{\end{eqnarray}}
\def\baqn{\begin{eqnarray*}}
\def\eaqn{\end{eqnarray*}}

\newcommand{\R}{\mathbb{R}}
\newcommand{\N}{\mathbb{N}}
\newcommand{\ball}{\mathbb{B}}

  \newcommand{\Multi}{\,\,\lower 1pt
     \hbox{$\overrightarrow{\longrightarrow}$}\,\,}
     \newcommand{\multi}{\,\,\lower 1pt
     \hbox{$\overrightarrow{\rightarrow}$}\,\,}

\def\image #1 (#2,#3) (echelle #4) #5{
\dimen2=#2
\dimen3=#3
\divide \dimen2 by 1000
\multiply \dimen2 by #4
\divide \dimen3 by 1000
\multiply \dimen3 by #4
\setbox1 =\vbox to \dimen2{\hsize=\dimen3\vfill\special{picture #1
scaled #4}}
\vbox{\hsize=\dimen3\box1\medskip\centerline{#5}}
}

\newtheorem{definition}{Definition}[section]
\newtheorem{prop}{Proposition}[section]
\newtheorem{thm}{Theorem}[section]

\newtheorem{lemma}{Lemma}[section]
\newtheorem{remark}{Remark}[section]
\newtheorem{example}{Example}[section]
\newtheorem{assumption}{Assumption}

\begin{document}
\author{Ba Khiet \sc Le\\
{\small \it Instituto de Ciencias de la Ingenier\'ia, Universidad de O'Higgins, Rancagua, Chile}\\
{\small  lebakhiet@gmail.com} \\
}
\maketitle
\begin{abstract}
In this paper, we introduce for the first time a class of state-dependent maximal monotone differential inclusions. Then the existence and uniqueness of solutions are obtained by using an implicit discretization scheme and a kind of hypo-monotonicity assumption respectively. In addition,  a characterization for nonsmooth Lyapunov pairs associated with such systems is provided. Our result can be applied to study state-dependent sweeping processes and Lur'e dynamical systems. It is new even the involved maximal monotone operators depend only on the time. 
\end{abstract}

\noindent {\bf Keywords} Differential inclusions, maximal monotone operators , state-dependent, Lur'e systems, sweeping processes \\
\noindent {\bf AMS subject classifications} 49J40, 47J22, 49K40, 37B25\\
\tableofcontents
\section{Introduction}
The paper is dedicated to studying the well-posedness and Lyapunov stability for a new class of state-dependent maximal monotone differential inclusions under perturbations as follows
\begin{equation}\label{main}
\left\{\begin{array}{l}
\dot{x}(t)\in f(t,x(t))-A_{t,x(t)} (x(t)), \;a.e. \; t\ge 0,  \\ \\
x(t_0)=x_0 \in {\rm dom}(A_{t_0,x_0}),
\end{array}\right.
\end{equation}
where $f: [0,+\infty) \times \R^n\to \R^n$ is a continuous function and $A_{t,x}: \R^n \to \R^n$ is a maximal monotone operator for each $(t,x)\in [0,+\infty) \times \R^n,$ i.e., the monotone operator $A$ depends on both time and state.   The classical maximal differential inclusions when $A$ is a fixed maximal monotone operator have been fruitfully studied in the literature, see for examples \cite{Aubin,Brezis}. There are also  various works for the case of time-dependence, i.e., $A_{t,x}\equiv A_t$ (see \cite{Kenmochi,Kunze0,Azzam,Pavel,Vladimirov} and the reference therein). Among important contributions are sweeping processes \cite{aht, Kunze,Le, M1,M2,M3,M4,t1,t2}, 
Skorohod problem \cite{TANAKA}, hysteresis operators \cite{KRASNOSEL?SKIMI} and recently Lur'e dynamical systems \cite{ahl,bg, bg1,bg2, cs,Leand,abc0,abc}. In particular, when $A_{t,x}\equiv N_{C(t)}$, the normal cone of a moving closed convex set, one obtains the sweeping processes
\begin{equation}
\left\{\begin{array}{l}
\dot{x}(t)\in f(t,x(t))-N_{C(t)} (x(t)), \;a.e. \; t\ge 0,  \\ \\
x(t_0)=x_0,
\end{array}\right.
\end{equation}
 which were introduced and thoroughly studied by J. J. Moreau \cite{M1,M2,M3,M4} in the seventies. Sweeping processes can be used for 
 granular material, quasi-static evolution in elastoplasticity, nonsmooth mechanics,
convex optimization, modeling of crowd motion, mathematical economics, switched electrical circuits. Another interesting application is Lur'e dynamical systems, which is of great interest in engineering, control theory and applied mathematics \cite{ahl, bg, bg2, cs, Liberzon,abc}.  The systems comprise an interconnection between a smooth ordinary differential equation  with a possibly set-valued feedback. 
 \begin{figure}[h!]
\begin{center}
\includegraphics[scale=0.8]{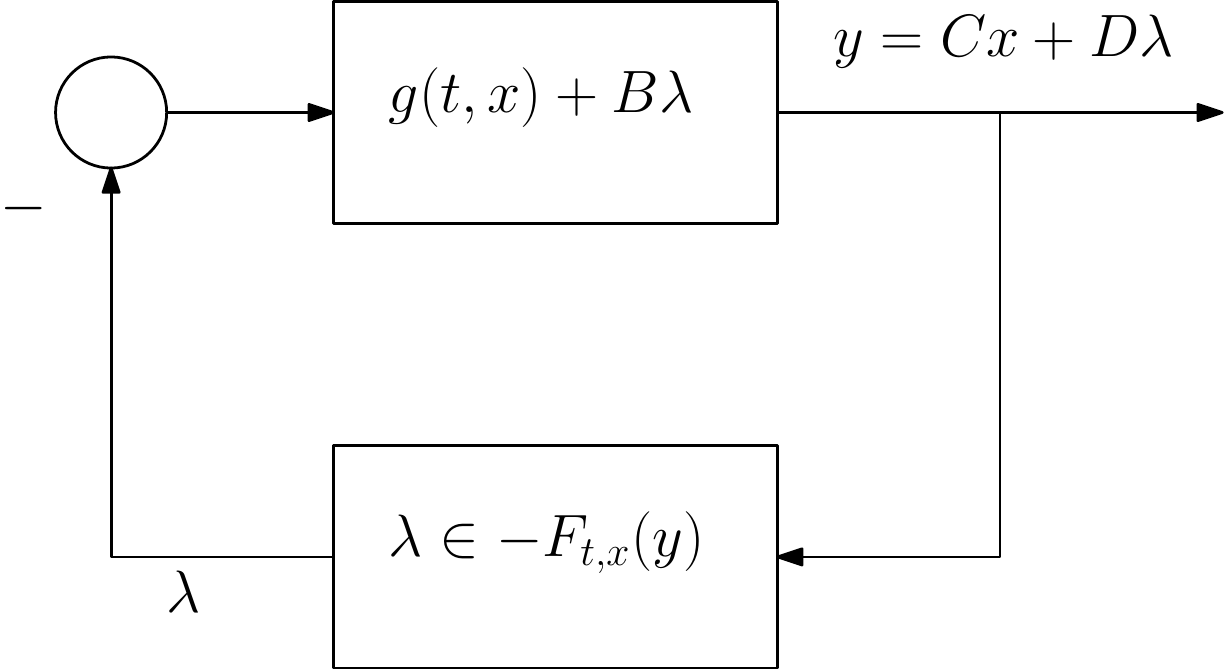}
\caption{ Lur'e  dynamical systems with state-dependent  set-valued feedback.}
\end{center}
\label{luref}
\end{figure}

In detail, let us consider the following systems 
\begin{equation}\label{Lure0}
\left\{\begin{array}{l}
\dot{x}(t) = g(t,x(t))+B\lambda(t)\; {\rm a.e.} \; t \ge 0;  \\ \\
y(t)=Cx(t)+D\lambda(t),\\ \\
\lambda(t) \in -F_{t,x(t)}(y(t)), \;t\ge 0;\\ \\
x(0)=x_0,
\end{array}\right.
\end{equation}
where    
\begin{itemize}
\item $g: [0,T] \times \R^n \to  \R^n $ is a continuous function;
\item  $ B:  \R^m\to  \R^n ,C: \R^n\to  \R^m$, $D : \R^m\to  \R^m$ are some  matrices;
\item  $F: [0,T] \times \R^n  \times\R^m \rightrightarrows \R^m$ is a set-valued mapping such that for each $(t,x) \in [0,T] \times \R^n$, the operator $F_{t,x}(\cdot)$ is maximal monotone. 
\end{itemize}
 After some simple computations, we can reduce $(\ref{Lure0})$ into the following  first order differential inclusions
\begin{equation}\label{Lurere}
\left\{\begin{array}{l}
\dot{x}(t)\in g(t,x(t))-B(F^{-1}_{t,y}+D)^{-1}Cx(t) \;a.e. \; t\ge 0,  \\ \\
x(t_0)=x_0,
\end{array}\right.
\end{equation}
which is a particular form of (\ref{main}). The case of stationary $F$ were considered in \cite{ahl,bg,bg1,bg2,cs} while the time-dependent case was studied  in \cite{abc} for $F_t \equiv N_{C(t)}$. The particular state-dependent case $F_{t,y}\equiv N_{C(t,y)}$ was  considered in \cite{Leand} recently.

To the best of our knowledge, there is still no research for general state-dependent maximal monotone differential inclusions,   which is our motivation to fill this gap.  On the other hand, Lyapunov stability of solutions through Lyapunov functions is always a nice property for any systems that people want to acquire. Smooth Lyapunov functions may limit applications due to the intrinsic nonsmoothness of many problems. Therefore, it is natural for us to consider nonsmooth Lyapunov pairs by using proximal analysis (see, e.g., \cite{aht,al,bg1,bg2,hv} and references therein for smooth as well as nonsmooth Lyapunov functions for the stationary case and time-dependent sweeping processes). Our result is new even the involved maximal monotone operator depends only on the time.

 The paper is organized as follows. In Section \ref{section2}, we recall some useful definitions and results which will be used later. Then the well-posedness of  problem (\ref{main}) is considered in Section \ref{section3} by using an implicit discretization scheme and a kind of hypo-monotonicity assumption. In Section 4, we provide a charactization for lower semi-continous Lyapunov pairs  for (\ref{main}) with some illustrative examples. The obtained results can be applied to study the state-dependent sweeping processes and Lur'e dynamical systems in Section 5.  Finally, some conclusions end the paper in Section 6. 
 \section{Mathematical backgrounds}
\label{section2}
Let us first introduce some notation that will be used in the sequel. Denote by $\langle\cdot,\cdot\rangle$ , $\|\cdot\|, \ball$ the  scalar product, the corresponding norm and the closed unit ball in Euclidean spaces. Let be given a closed, convex set $K\subset \R^n$. The distance and the projection from a point $s$ to $K$ are defined respectively by 
$${ d}(s,K):=\inf_{x\in K} \|s-x\|, \;\;{\rm proj}(s,K):=x \in K \;\;{\rm such \;that \;} { d}(s,K)= \|s-x\|.$$
The minimal norm element of $K$ is  defined by 
$$
K^0:= {\rm proj}(0; K).
$$
The normal cone of $K$ is given by
$$
N(K,x)=\{ x^*: \langle x^*,y-x  \rangle \le 0,\;\; \forall \; y\in K \}.
$$
The Hausdorff distance between two closed, convex sets $K_1, K_2$ is given by 
$$d_H(K_1,K_2):=\max\{\sup_{x_1\in K_1} d(x_1,K_2), \sup_{x_2\in K_2} d(x_2,K_1)\}.$$
Now let $K$ be a closed subset of $\R^n$. The proximal normal cone, the limiting normal cone and the Clarke normal cone of $K$ at $x$ are defined \cite{Clarke,BM} respectively as  
$$
N^P(K,x):=\{\xi \in \R^n: \exists \delta>0\;s.t.\;\langle \xi, y-x \rangle \le \delta \|y-x\|^2\;\;{\rm for\;all\;}y\in K\},
$$ 
$$
N^L(K,x):=\{\xi\in \R^n:\exists\; \xi_n\to \xi\;{\rm\;and}\; \;\xi_n\in N^P(K,x_n), x_n\to x\;{\rm in}\; K\},\\
$$
and
$$N^C(K,x):=\overline{{\rm co}}N^L(K,x).$$
It is easy to see that if $K$ is also convex then $N^P(K,x)= N^L(K,x) = N^C(K,x)=N(K,x).$
\begin{definition}
Let $\varphi: \R^n \to {\R}\cup \{+\infty\}$ be a proper lower semi-continuous function and $x\in \R^n$ at which $\varphi$ is finite. The proximal subdifferential, singular subdifferential and Clarke subdifferential of $\varphi$ at $x$ are defined respectively by 
$$
\partial^P\varphi(x):=\{\xi\in \R^n:  (\xi,-1)\in N^P_{{\rm epi}\;\varphi}(x,\varphi(x))\},
$$
$$
\partial^\infty\varphi(x):=\{\xi\in \R^n:  (\xi,0)\in N^L_{{\rm epi}\;\varphi}(x,\varphi(x))\},
$$
and
$$
\partial^C\varphi(x):=\{\xi\in \R^n:  (\xi,-1)\in N^C_{{\rm epi}\;\varphi}(x,\varphi(x))\}.
$$
\end{definition}
It is known that $\xi\in\partial^\infty\varphi(x)$ if and only if there exist  sequences $(\alpha_k)_{k\in \N}\subset \R_+$, $(x_k)_{k\in\N}$, $(\xi_k)_{k\in \N}$ such that $\alpha_k\to 0^+$, $x_k \rightarrow_\varphi x$, $\xi_k\in \partial^P\varphi(x_k)$ and 
$\alpha_k\xi_k\to \xi$ (see, e.g., \cite{BM}).

\begin{definition}
\noindent  A matrix $P\in \R^{n\times n}$  is called
 \begin{itemize}
\item  $positive$ $semidefinite$ if for all $x\in {\R^n},$ we have 
$$\langle Px,x \rangle \ge 0;$$
\item $positive$ $definite$ if there exists $\alpha>0$ such that for all $x\in {\R^n},$ we have 
$$
\langle Px,x \rangle \ge \alpha\|x\|^2.
$$
\end{itemize}
\end{definition}
We have the following fundamental lemma.
\begin{lemma}\label{estiD}
Let $D$ be a positive semidefinite  matrix. Then there exists some constant $c_1>0$ such that for all $x\in  {\rm rge}(D+D^T)$, we have:
\beq \label{constantd}
\langle Dx, x \rangle \ge c_1 \Vert x \Vert^2.
\eeq
\end{lemma}
\begin{proof}
If $D$ is non-zero then $c_1$ can be chosen as the smallest positive eigenvalue of $D+D^T$.
\end{proof}
\begin{definition}
A set-valued mapping $F: \R^n \rightrightarrows \R^n$ is called $\it{monotone}$ if for all $x,y\in \R^n,x^*\in F(x),y^*\in F(y)$, one has $\langle x^*-y^*,x-y\rangle \ge 0.$ In addition, it  is called $\it{maximal \;monotone}$ if there is no monotone  operator $G$ such that the graph of $F$ is contained strictly in  the graph of $G,$ i.e., if 
$$
\langle x^*-y^*,x-y\rangle \ge 0,
$$
for all $(y,y^*)$ in the graph of $F$, then $x^*\in F(x)$.
\end{definition}
 \begin{prop}{\rm (\cite{Aubin,Brezis})}\label{Yosida}
Let $H$ be a Hilbert   space, $F:  {H} \rightrightarrows  {H}$ be a maximal monotone operator and let $\lambda>0$. Then\\
$1)$ the \emph{resolvent} of $F$  defined by $J_F^\lambda:=(I+\lambda F)^{-1} $ is a non-expansive and single-valued map from $ {H}$ to $ {H}$.\\
$2)$ the \emph{Yosida approximation} of $F$ defined by $F^\lambda:=\frac{1}{\lambda}(I-J_F^\lambda)=(\lambda I+F^{-1})^{-1}$ satisfies\\
\indent i) for all $x\in {H}$, $F^\lambda(x)\in F(J_F^\lambda x)$ ,\\
\indent ii) $F_\lambda$ is Lipschitz continuous with constant $\frac{1}{\lambda}$ and also maximal monotone.\\
\indent iii) If $x\in {\rm dom}(F)$, then $\|F^\lambda x\| \le \|F^0x\|$, where $F^0x$ is the element of $Fx$ of minimal norm.\\
\end{prop}
\noindent  Let us recall Minty's Theorem in the setting of Hilbert  spaces  (see \cite{Aubin,Brezis}).
\begin{prop}\label{minty}
Let $H$ be a Hilbert   space.  Let $F: H\rightrightarrows H$ be a monotone operator. Then $F$ is maximal monotone if and only if ${\rm rge}(F+I)=H.$
\end{prop}
Let be given two maximal monotone operators $F_1$ and
$F_2$, we recall  the definition of  pseudo-distance between  $F_1$ and $F_2$ introduced by Vladimirov \cite{Vladimirov} as follows
$$
{\rm dis}( F_1, F_2):=\sup\Big\{ \frac{\langle \eta_1 -\eta_2,z_2-z_1\rangle}{1+|\eta_1|+|\eta_2|}:  \eta_i\in F(z_i), z_i\in {\rm dom}\; (F_i), i=1,2\Big\}.
$$
\begin{lemma}\label{hauslem}
\cite{Vladimirov} If $ F_i=N_{A_i}$  where $A_i$ is a  closed convex set  $(i=1,2)$ then 
$$
{\rm dis}( F_1,  F_2)=d_H(A_1,A_2).
$$
\end{lemma}

\begin{lemma}\label{esti2mm} \cite{Kunze}
Let $F_1$, $F_2$ be two maximal monotone operators. For $\lambda>0, \delta >0$ and $x\in {\rm dom}(F_1)$, we have
\baqn
\Vert x - J^\lambda_{F_2}(x)\Vert &\le& \lambda \Vert F_1^0 x \Vert + {\rm dis}(F_1, F_2)+\sqrt{\lambda(1+\Vert F_1^0 x \Vert){\rm dis}(F_1, F_2)}\\
&\le& \lambda\Vert F_1^0 x \Vert + {\rm dis}(F_1, F_2)+ (\delta {\rm dis}(F_1, F_2)+\frac{\lambda(1+\Vert F_1^0 x \Vert)}{4\delta})\\
&\le& \frac{\lambda(1+(4\delta+1)\Vert F_1^0 x \Vert)}{4\delta})+(1+\delta){\rm dis}(F_1, F_2).
\eaqn
\end{lemma}
\begin{lemma}\label{closemm} \cite{Kunze}
Let $F_n$ be a sequence of maximal monotone operators in a Hilbert space $H$ such that ${\rm dis}(F_n,F)\to 0$ as $n\to+\infty$ for some maximal monotone operator $F$. Suppose that $x_n\in {\rm dom}(F_n)$ with $x_n\to x$ and that $y_n\in F_n(x_n)$ with $y_n\to y$ weakly for some $x, y\in H$. Then $x\in {\rm dom}(F)$ and $y\in F(x)$.
\end{lemma}
Let us end-up this section by recalling some versions of Gronwall's inequality.

\begin{lemma}\label{Gronwalldis}
Let $\alpha>0$ and $(u_n)$, $(\beta_n)$ be non-negative sequences satisfying 
\beq
u_n\le \alpha+\sum_{k=0}^{n-1}\beta_k u_k \;\;\forall n= 0,1,2,\ldots\;\; ({\rm with} \;\beta_{-1}:=0).
\eeq
Then, for all $n$, we have
$$u_n\le \alpha \;{\rm exp}\Big(\sum_{k=0}^{n-1}\beta_k\Big).$$
\end{lemma}

\begin{lemma}\label{gronwall}
Let $T>0$ be given and $a(\cdot),b(\cdot)\in L^1([0,T];\R)$ with $b(t)\ge 0$ for almost all $t\in [0,T].$ Let an absolutely continuous function $w: [0,T]\to \R_+$ satisfy
\beq
(1-\alpha)w'(t)\le a(t)w(t)+b(t)w^\alpha(t),\;\; a. e. \;t\in [0,T]
\eeq
where $0\le \alpha<1$. Then for all $t\in [0,T]$, we have 
\beq
w^{1-\alpha}(t)\le w^{1-\alpha}(0){\rm exp}\Big(\int_{0}^t a(\tau)d\tau\Big)+\int_{0}^t{\rm exp}\Big(\int_{s}^t a(\tau)d\tau\Big)b(s)ds.
\eeq
\end{lemma}
\section{Well-posedness of the problem}\label{section3}
In this section, the existence and uniqueness of solutions for problem (\ref{main}) are studied by using an implicit approximation scheme inspired by \cite{Leand} and a kind of hypo-monotonicity assumption repsectively. 
Let be given arbitrary $T>0$. First, we propose  followings assumptions.
\begin{assumption}
 For every $t\in [0,T]$ and $x\in \R^n$, the operator $A_{t,x}: \R^n \rightrightarrows \R^n $ be a maximal monotone operator such that\\
 \noindent $(1.1)$ there exists non-negative constants $L_{1}, L_2$ with $ L_2<1$ such that 
\beq
{\rm dis}(A_{t,x}, A_{s,y})\le  L_1 \vert t-s \vert+L_2\Vert x-y \Vert,\;\;\forall \;t, s \in  [0,T].
\eeq
$(1.2)$ there exists  $c_A>0$ such that
\beq
\Vert A^0_{t,x}(y) \Vert \le c_A(1+ \Vert x \Vert+\Vert y \Vert), \; t\in [0,T], x, y\in \R^n .
\eeq
\end{assumption}

\begin{assumption}
Let  $f: [0,T] \times \R^n \to \R^n$ be a continuous function. In addition, suppose that there exists $c_f>0$ such that
\beq
\Vert f(t,x) \Vert \le c_f(1+\Vert x \Vert), \;\forall \; t\in [0,T], x\in \R^n.
\eeq
\end{assumption}
We define the admissible set for (\ref{main}) as follows:
\beq
\mathcal{A}_1:=\{(t_0,x_0): x_0\in {\rm dom}(A_{t_0,x_0})\}. 
\eeq
\begin{thm} (Existence)\label{exist}
Let Assumptions 1, 2 hold. 
Then for all $(t_0,x_0) \in \mathcal{A}_1$ with $0\le t_0\le T$, the problem
\beq\nonumber
\dot{x}(t)\in -A_{t,x(t)} x(t)+f(t,x(t)),\;\;a.e. \;t\in [t_0,T], x(t_0)=x_0,
\eeq
has a  solution $x(\cdot)$. In addition, we have 
\beq
\Vert \dot{x}(t) \Vert \le m(x_0), \;\;a.e.\;  \;t\in [t_0,t_0+1],
\eeq
where $m(x_0)>0$ is defined in (\ref{mx0}) depending only on $x_0, L_1, L_2, c_A, c_f$ and $m(\cdot)$ is a continuous function w.r.t $x_0$. Therefore
\beq
\Vert x(t) \Vert \le \Vert x_0 \Vert+ m(x_0), \;\;\forall \;t\in [t_0,t_0+1],
\eeq
\end{thm}
\begin{proof} We use an implicit scheme to approximate problem (\ref{main}). In details, 
let $T'=T-t_0$ and for each given positive integer $n$, we set $h_n=T'/n$ and $t^n_i=t_0+ih$ for $0\le i\le n.$  We can construct the sequence $({x}^n_i)_{0\le i\le n}$ with $x^n_0=x_0$ as follows: 
 \begin{equation}\label{discrete}
\left\{\begin{array}{l}
y^n_i= {x}^n_{i}+h_nf(t^n_i,x^n_i), \\ \\
{x}^n_{i+1}\in y_i^n -h_n A_{t^n_{i+1},  x^n_{i}}(x^n_{i+1}).
\end{array}\right.
\end{equation}
Indeed, we can compute ${x}^n_{i+1}$ by 
$$
x^n_{i+1}=(I+h_n A_{t^n_{i+1},  x^n_{i}})^{-1}(y_i^n)=J^{h_n}_{A_{t^n_{i+1},  x^n_{i}}}(y^n_i).
$$
Thus we obtain the following algorithm which is well-defined.\\

\noindent $\mathbf{Algorithm}$

\noindent \texttt{Initialization.} Let $x^n_0:=x_0, y^n_0:=x^n_{0}+h_nf(t^n_0,x^n_0).$ \\

\noindent\texttt{Iteration.} For the  current points $x^n_i$, we  compute 
\beq\label{proje}
y^n_i:=x^n_{i}+h_nf(t^n_i,x^n_i),\;\;{\rm and}\;x^n_{i+1}:=J^{h_n}_{A_{t^n_{i+1}, x^n_{i}}}(y^n_i).
\eeq
From (\ref{proje}), we have
\baq\nonumber
\Vert x^n_{i+1}-x^n_{i}\Vert&=&\Vert J^{h_n}_{A_{t^n_{i+1},x^n_i}}(y^n_i)-x^n_{i}\Vert \\
&\le&\Vert J^{h_n}_{A_{t^n_{i+1},x^n_i}}(y^n_i)- J^{h_n}_{A_{t^n_{i+1},x^n_i}}(x^n_i) \Vert +\Vert J^{h_n}_{A_{t^n_{i+1},x^n_i}}(x^n_i)-x^n_i\Vert.
\label{est1}
\eaq
Since $J^{h_n}_{A_{t^n_{i+1},x^n_i}}$ is non-expansive, one has
\baq \nonumber
\Vert J^{h_n}_{A_{t^n_{i+1},x^n_i}}(y^n_i)- J^{h_n}_{A_{t^n_{i+1},x^n_i}}(x^n_i) \Vert&\le& \Vert  y^n_i- x^n_i\Vert \\\label{est2}
&\le& h_n \Vert f(t^n_i,x^n_i)\Vert  \le h_n c_f( 1+\Vert x^n_i \Vert).
\eaq
Since $L_2<1$, we can always choose some constant $\delta>0$ such that 
\beq\label{lktilde}
\tilde{L}_2:=(1+\delta)L_{2} <1.
\eeq
 Note that $x^n_i\in {\rm dom}(A_{t^n_{i},x^n_{i-1}})$ for $i=0,..,n$ with $x^n_{-1}:=x^n_0$, by using Lemma \ref{esti2mm} and Assumption 1.1,  we have  
\baq\nonumber
\Vert J^{h_n}_{A_{t^n_{i+1},x^n_i}}(x^n_i)-x^n_i\Vert&\le&  h_n\frac{1+(4\delta+1)\Vert A^0_{t^n_{i},x^n_{i-1}}(x^n_i)\Vert}{4\delta}+(1+\delta){\rm dis}(A_{t^n_{i+1},x^n_i}, A_{t^n_{i},x^n_{i-1}})\\\nonumber
&\le& h_n  \frac{1+(4\delta+1) c_A(1+\Vert x^n_i \Vert + \Vert  x^n_{i-1} \Vert)}{4\delta}\\
&+&{(1+\delta) L_{1}}h_n+{(1+\delta)L_{2} } \Vert x^n_i-x^n_{i-1}\Vert).
\label{est3}
\eaq
From (\ref{est1}), (\ref{est2}), (\ref{lktilde}) and (\ref{est3}), one has
\baq\nonumber
\Vert x^n_{i+1}-x^n_{i}\Vert &\le& h_nc_1(1+\Vert x^n_{i+1}\Vert +\Vert x^n_i \Vert +\Vert  x^n_{i-1} \Vert) )\\
&+&\tilde{L}_2 \Vert x^n_i-x^n_{i-1}\Vert)
\eaq
where $\tilde{L}_2<1$ and 
$$
c_1:=c_f+\frac{1+(4\delta+1) c_A}{4\delta}+(1+\delta) L_{1}.
$$
Since $x^n_{-1}:=x^n_0$,  we obtain  
\baq\label{estderi}
\Vert x^n_{i+1}-x^n_{i}\Vert &\le& h_nc_1 \sum_{j=0}^i \tilde{L}^j_2(1+  \Vert x^n_{i-j+1} \Vert +  \Vert x^n_{i-j} \Vert +\Vert x^n_{i-j-1} \Vert )\\
&\le&  h_nc_1(\frac{1}{1-\tilde{L}_2}+ \sum_{j=0}^i \tilde{L}^j_2( \Vert x^n_{i-j+1} \Vert +  \Vert x^n_{i-j} \Vert +\Vert x^n_{i-j-1} \Vert ).\nonumber
\eaq
Thus  
\baqn
\Vert x^n_{i+1}-x^n_{0}\Vert &\le&  \sum_{j=0}^i  \Vert x^n_{j+1}-x^n_{j}\Vert \\
&\le& {h_nc_1}(\frac{i+1}{ {1-\tilde{L}_2}}+ \Vert x^n_{i+1}\Vert+ 3\sum_{j=0}^i \tilde{L}^j_2\sum_{j=0}^i \Vert x^n_j \Vert)\\
&\le&  \frac{c_1T'}{ {1-\tilde{L}_2}}+{h_nc_1}\Vert x^n_{i+1}\Vert +\frac{3h_nc_1}{ {1-\tilde{L}_2}}\sum_{j=0}^i \Vert x^n_j \Vert.
\eaqn
We can choose $n$ large enough such that $h_nc_1<1/2$. Then we  have 
$$
\Vert x^n_{i+1}\Vert\le c_2+c_3 h_n\sum_{j=0}^i \Vert x^n_j \Vert,
$$
where
$$
c_2:=2\Vert x_0\Vert+\frac{2c_1T'}{ {1-\tilde{L}_2}}, \;\;c_3:=\frac{6c_1}{ {1-\tilde{L}_2}}.
$$
Thus one has 
\beq
\Vert x^n_{i+1} \Vert \le M_1:=c_2 e^{c_3 T'},\;i=0, 1, \ldots,n-1,
\eeq
by  using the discrete Gronwall's inequality in Lemma \ref{Gronwalldis}.
From (\ref{estderi}), we deduce that  
\beq\label{estv}
\Vert\frac{x^n_{i+1}-x^n_{i}}{h_n}\Vert \le \frac{c_1(1+3M_1)}{1-\tilde{L}_2}:=M_2 .
\eeq
Now let us construct the sequences of functions $(x_n(\cdot))_n,$  $ (\theta_n(\cdot))_n,$ $(\eta_n(\cdot))_n,$ on $[t_0,T]$   as follows: for $0\le i \le n-1$, on $[t^n_i,t^n_{i+1})$ , we define
\beq\label{funx}
x_n(t):=x^n_i+\frac{x^n_{i+1}-x^n_i}{h_n}(t-t^n_i), 
\eeq
and 
\beq\label{funtheta}
\theta_n(t):=t^n_i,\;\;\eta_n(t):=t^n_{i+1}.
 \eeq
Thanks to (\ref{estv}), for all $t\in (t^n_{i},t^n_{i+1})$, we have 
$$\|\dot{x}_n(t)\|=\|\frac{x^n_{i+1}-x^n_i}{h_n}\|\le  M_2,$$
and 
 \beq\label{idapro}
 \sup_{t\in  [t_0,T]}\{|\theta_n(t)-t|,|\eta_n(t)-t|\}  \le h_n\to 0\; {\rm as}\;\;n\to +\infty.
 \eeq
  Therefore  $\big(x_n(\cdot)\big)_n$ is uniformly bounded and equi-Lipschitz continuous.  As a consequence of   Arzel\`a--Ascoli theorem, one can find a Lipschitz continuous  function $x(\cdot): [t_0,T]\to \R^n$  and a subsequence, still denoted by  $\big(x_n(\cdot)\big)_n$, satisfying:  
\begin{itemize}
\item $x_n(\cdot)$ converges strongly to $x(\cdot)$ in $\mathcal{C}([t_0,T];\R^n)$;
\item $\dot{x}_n(\cdot)$ converges weakly to $\dot{x}(\cdot)$ in $L^{2}([t_0,T];\R^n)$.
\end{itemize}
  Particularly, we have $x(0)=x_0.$ From (\ref{discrete}), (\ref{funx}) and (\ref{funtheta}), we obtain
\baq\label{appro}
\dot{x}_n(t)&\in&f(\theta_n(t), x_n(\theta_n(t)))
- A_{\eta_n(t), x_n(\theta_n(t)}(x_n(\eta_n(t))).
\eaq
 For each positive integer $n$,   let us define the operators $ \mathcal{A}, \mathcal{A}_n: L^{2}([t_0,T];\R^n) \to L^{2}([t_0,T];\R^n)$ as follows
$$
z^*\in \mathcal{A}(z) \Leftrightarrow z^*(t)\in A_{t, x(t)}(z(t)) \;a.e. \; t\in [t_0,T],
$$
and 
$$
z^*\in \mathcal{A}_n(z) \Leftrightarrow z^*(t)\in A_{\eta_n(t), x_n(\theta_n(t))}(z(t)) \;a.e. \; t\in [t_0,T].
$$
Using Minty's theorem, it is easy to see that  $\mathcal{A}_n, \mathcal{A}$ are maximal monotone operators since   $A_{t, x(t)}$ and $A_{\eta_n(t), x_n(\theta_n(t))}$ are maximal monotone for each $t\in [t_0,T]$. Furthermore, we have 
\baqn
&&{\rm dis}( \mathcal{A}_n,  \mathcal{A})\\
&=&\sup\Big\{\frac{  \int_{t_0}^T\langle z^*_n(t) -z^*(t),z_n(t)-z(t)\rangle dt}{1+\|z^*_n\|_{L^2}+\|z^*\|_{L^2}}:  z^*_n\in   \mathcal{A}_n(z_n), z^*\in     \mathcal{A}(z)\Big\}\\\nonumber\\
&\le &\sup\Big\{\frac{  \int_{t_0}^T{\rm dis}(A_{\eta_n(t), x_n(\theta_n(t))},A_{t,x(t)})(1+\|z^*_n(t) \|+\|z^*(t) \| ) dt}{1+\|z^*_n\|_{L^2}+\|z^*\|_{L^2}}:  z^*_n\in   \mathcal{A}_n(z_n), z^*\in     \mathcal{A}(z)\Big\}\\\nonumber
&& ({\rm \;using \;the\;definition \; of} \;{\rm dis}(A_{\eta_n(t), x_n(\theta_n(t))},A_{t,x(t)}))\\
&\le &\sup\Big\{\frac{  \int_{t_0}^T(L_{1}\vert\eta_n(t)-t \vert +L_{2}\Vert x_n(\theta_n(t)-x(t) \Vert)(1+\|z^*_n(t) \|+\|z^*(t) \|) dt}{1+\|z^*_n\|_{L^2}+\|z^*\|_{L^2}}: \\\nonumber
&& z^*_n\in   \mathcal{A}_n(z_n), z^*\in     \mathcal{A}(z)\Big\}\\\nonumber
&& {\rm (using \;Assumption\; 1\;)} \\
&\le &(L_{1} \Vert\eta_n-I \Vert_{L^2} +L_{2}\Vert x_n\circ \theta_n-x \Vert_{L^2})\sup\Big\{ \frac{1+\|z^*_n \|_{L^2}+\|z^* \|_{L^2}}{1+\|z^*_n \|_{L^2}+\|z^* \|_{L^2}}:\\\nonumber
&&z^*_n\in   \mathcal{A}_n(z_n), z^*\in     \mathcal{A}(z)\Big\}\\\nonumber
&= &L_{1} \Vert\eta_n-I \Vert_{L^2} +L_{2}\Vert x_n\circ \theta_n-x \Vert_{L^2} \to 0,
\eaqn
as $n\to +\infty.$

 Note that $\dot{x}_n$ converges weakly to $\dot{x}$ in $L^{2}([t_0,T];\R^n)$. Thus  Assumption 2 allows us to deduce  that
$$
\dot{x}_n - f(\theta_n(\cdot), x_n\circ\theta_n(\cdot)) \to \dot{x} - f(\cdot,x)
$$
 weakly in $L^{2}([t_0,T];\R^n)$.
 In addition,  $x_n\circ \eta_n$ converges strongly $x$ in $L^{2}([t_0,T];\R^n)$.  Using Lemma \ref{closemm} and (\ref{appro}), we have  
 \beq
 \dot{x}- f(\cdot,x)\in -\mathcal{A}(x),
 \eeq
 or equivalently
 \beq
 \dot{x}(t)- f(t,x(t))\in -{A}_{t,x(t)}(x(t)),\;\;a.e.\; t\in [t_0,T],
 \eeq
 which shows that $x(\cdot)$ is a solution of (\ref{main}). Note that in (\ref{estv}) if we replace $T'=T-t_0$ by $1$, we can obtain that 
 \beq
\Vert \dot{x}(t) \Vert \le m(x_0), \;\;a.e.\;  \;t\in [t_0,t_0+1],
\eeq
where
\baq\nonumber
m(x_0)&:=&c_2 e^{c_3}=(2\Vert x_0\Vert+\frac{2c_1}{ {1-\tilde{L}_2}}){\rm exp}(\frac{6c_1}{ {1-\tilde{L}_2}})\\
&=& (2\Vert x_0\Vert+\frac{2c_1}{ {1-(1+\delta){L}_2}}){\rm exp}(\frac{6c_1}{ {1-(1+\delta){L}_2}}) \label{mx0}
\eaq
and
$$
c_1:=c_f+\frac{1+(4\delta+1) c_A}{4\delta}+(1+\delta) L_{1}.
$$
  The proof is completed. 
\end{proof}
\begin{remark}
$(i)$The result in Theorem \ref{exist} is still valid for infinite dimensional Hilbert spaces with some  additional compactness assumption. Let us note that even the existence of solutions in infinite dimensional spaces for state-dependent sweeping processes, a particular case of (\ref{main}), without any compactness assumption  is still an open question.\\
$(ii)$ The single-valued perturbation $f$ can be replaced standardly by a set-valued upper semi-continuous mapping with convex weakly compact values satisfying some linear growth condition.
\end{remark}
\begin{thm} (Uniqueness) \label{uniq}
Let all assumptions in Theorem \ref{exist} hold. In addition, suppose that $f$ is Lipschitz continuous on bounded sets w.r.t the second variable and $A$ is hypo-monotone on bounded sets, in the sense that,  for given $M>0$ there exist $k_M, l_M>0$ such that for all $t\in [0,T], x_i\in \R^n \cap M\ball, x^*_i\in A_{t,x_i}(x_i)$, $i=1,2$ we have 
\beq \label{hypo}
\langle x^*_1-x^*_2, x_1-x_2 \rangle \ge -k_M \Vert x_1-x_2 \Vert^2,
\eeq
and 
$$
\Vert f(t,x_1)-f(t,x_2)\Vert\le l_M \Vert x_1-x_2 \Vert.
$$
Then for each $(t_0,x_0)\in \mathcal{A}_1$,  problem (\ref{main}) has a unique solution on $[t_0,T]$.
\end{thm}
\begin{proof}
Let $x_1(\cdot), x_2(\cdot)$ be two solutions of (\ref{main}) with the same initial conditions $x_1(t_0)=x_2(t_0)=x_0$. Then for almost all $t\in [t_0,T],$ one has
\begin{equation}
\left\{\begin{array}{l}
\dot{x}_1(t)\in f(t,x_1(t))-A_{t,x_1(t)}({x}_1(t)),\\ \\
\dot{x}_2(t)\in f(t,x_2(t))-A_{t,x_2(t)}({x}_2(t)).
\end{array}\right.
\end{equation}
Using the hypo-monotonicity of $A$ and Lipschitz continuity of $f$, we have 
\baqn
\langle \dot{x}_1(t)-\dot{x}_2(t), x_1(t)-x_2(t)\rangle &\le& \langle  f(t,x_1(t))- f(t,x_2(t)), x_1(t)-x_2(t)\rangle+k_M \Vert x_1(t)-x_2(t) \Vert^2\\
&\le& (l_M+k_M)\Vert x_1(t)-x_2(t) \Vert^2,
\eaqn
where $k_M, l_M$ are defined in (\ref{hypo}) and  $M>0$ is a constant such that 
$$\max\{\sup_{t\in [t_0,T]} \Vert x_1(t)\Vert,  \sup_{t\in [t_0,T]} \Vert x_2(t)\Vert \}\le M.$$ Then we have 
$$
\frac{d}{dt} \Vert x_1(t)-x_2(t)\Vert^2\le 2(l_M+k_M)\Vert x_1(t)-x_2(t) \Vert^2 \;\;\;a.e.\;\; t\in [t_0,T],
$$
and we obtain the conclusion  by using the continuous Gronwall's inequality in Lemma \ref{gronwall}. 
\end{proof}
\begin{remark}\label{rm1}
Let us provide some cases such that $A$ is hypo-monotone. \\
(i) Clearly if $A$ dependent only on the time, i.e. $A_{t,x}\equiv A_t$, then $A$ is hypo-monotone. \\
(ii) We show that the property also holds if $A_{t,x}(\cdot)=B_t(\cdot+\alpha x)$ where $\alpha > -1$ and $B_t$ is a maximal monotone operator for each $t\ge 0$. By using Minty's theorem, it is easy to see that $A_{t,x}$ is maximal monotone. In addition, for all $t\in [0,T], x_i\in \R^n, x^*_i\in A_{t,x_i}(x_i)$, $i=1,2$ we have
\beq
\langle x^*_1-x^*_2, x_1-x_2 \rangle= \frac{1}{1+\alpha} \langle x^*_1-x^*_2, (\alpha+1)(x_1-x_2) \rangle \ge 0. 
\eeq
(iii) If $A_{t,x}=(B^{-1}_{t,x}+D)^{-1}$ where $D$ is a positive semidefinite matrix, $B_{t,x}(\cdot)=C_t(\cdot+g(t,x) )$ and  $C_t: \R^n \rightrightarrows \R^n$ is a maximal monotone operator for each $t\in  [0,T]$ and $g:   [0,T] \times \R^n \to {\rm rge}(D+D^T)$ is a Lipschitz continuous in bounded sets w.r.t the second variable. This particular form of $A$ appears widely in Lur'e dynamical systems (see, e.g., \cite{ahl,bg1,bg2,cs,abc}).  Let us first show that $B_{t,x}$ is a maximal monotone operator and then so is $A_{t,x}$ since $D$ is a maximal monotone operator with full domain.
\begin{itemize}
\item Monotonicity: Let $y_i^*\in B_{t,x}(y_i)=C_t(y_i+g(t,x))$, $i=1,2$. Then
$$
\langle y^*_1-y^*_2, y_1-y_2 \rangle= \langle y^*_1-y^*_2, (y_1+g(t,x))-(y_2+g(t,x)) \rangle\ge 0.
$$
\item Maximality: It is equivalent to show that $B^{-1}_{t,x}$ is maximal.  Let $(y,y^*)\in \R^{2n}$, suppose that 
$$
\langle y^*-z^*, y-z \rangle \ge 0,
$$
for all $(z,z^*)$ satisfying $z\in B^{-1}_{t,x}(z^*)\Leftrightarrow z^*\in C_t(z+g(t,x)) \Leftrightarrow z+g(t,x) \in C_t^{-1}(z^*)$. We want to prove that $y\in B^{-1}_{t,x}(y^*)$. Indeed, we have 
$$
\langle y^*-z^*, (y+g(t,x))-(z+g(t,x)) \rangle \ge 0.
$$
Since $C_t^{-1}$ is maximal monotone, we must have $y+g(t,x))\in C_t^{-1}(y^*)$, or equivalently, $y\in B^{-1}_{t,x}(y^*)$. Consequently we obtain the maximality of $B_{t,x}$.
\end{itemize}
It remains to check that $A$ is hypo-monotone. For all $t\in [0,T], x_i\in \R^n \cap M\ball, x^*_i\in A_{t,x_i}(x_i)$, $i=1,2$ we have
\baqn
&&x^*_i\in A_{t,x_i}(x_i)=(B^{-1}_{t,x_i}+D)^{-1}(x_i)\\
&\Leftrightarrow& x^*_i \in B_{t,x_i}(x_i-Dx^*_i)=C_t(x_i-Dx^*_i+g(t,x_i)).
\eaqn
From the monotonicity of $C_t$, we have 
\baqn
\langle x^*_1-x^*_2, x_1-x_2 \rangle &\ge& \langle D(x^*_1-x^*_2), x^*_1-x^*_2 \rangle-\langle x^*_1-x^*_2,g(t,x_1)-g(t,x_2)\rangle\\
&\ge& c_1 \Vert x^{*im}_1-x^{*im}_2 \Vert^2 - L_g \Vert  x^{*im}_1-x^{*im}_2 \Vert \Vert x_1-x_2\Vert \;({\rm since\;rge}(g)\subset {\rm rge}(D+D^T))\\
&\ge& -\frac{L_g^2}{4c_1} \Vert x_1-x_2\Vert^2\;({\rm using \; the \; inequality \;} a^2+b^2\ge 2ab,\; a, b \in \R),
\eaqn
where $L_g$ is the Lipschitz constant of $g$  w.r.t the second variable in $M\ball$, $c_1$ is defined in Lemma \ref{estiD} and $x^{*im}$ denotes the projection of $x^*$ onto ${\rm rge}(D+D^T)$. Consequently, the conclusion follows. 
\end{remark}
\section{ Stability analysis by using nonsmooth Lyapunov pairs }
\label{section4}
In this section, we want to provide a characterization for lower semi-continuous Lyapunov pairs associated with problem (\ref{main}) by using proximal analysis. From here, we suppose that all assumptions of Theorem \ref{uniq} are satisfied, then for each
 $(t_0, x_0)\in\mathcal{A}_1$, the problem (\ref{main}) 
has a unique solution $x(\cdot)$ defined on $[t_0,+\infty)$.  
\noindent Next we recall the definition of a Lyapunov pair associated with problem (\ref{main}). Denote by
 $$\Gamma([0,+\infty)\times \R^n):=\{\varphi: [0,+\infty)\times \R^n\to {\R} \cup \{+\infty\}| \; \varphi \;{\rm is \;proper \;and\; lsc}\},$$
 and
 $$
 \Gamma_+([0,+\infty)\times \R^n):=\{\varphi: [0,+\infty)\times \R^n\to {\R}_+ \cup \{+\infty\}| \; \varphi \;{\rm is \;proper \;and\; lsc}\}.
 $$
\begin{definition}
Let $V\in \Gamma([0,+\infty)\times \R^n)$, $W\in \Gamma_+([0,+\infty)\times \R^n)$ and  $a\ge 0.$ We say that $(V, W)$ is an $a-$Lyapunov pair for problem $(\ref{main})$ if for all $(t_0, x_0) \in \mathcal{A}_1$ we have
\begin{equation}\label{501}
e^{a(t-t_0)}V\big(t,x(t; t_0;x_0)\big) +\int^t_{t_0}W\big(\tau,x(\tau; t_0;x_0)\big)d\tau \le V(t_0,x_0) \,\,\,\text{for all}\,\, t \ge t_0\ge 0,
\end{equation}
where $x(t; t_0, x_0)$ denotes the unique solution of problem $(\ref{main})$ satisfying   $x(t_0)=x_0.$
 If $a=0,$ then $(V, W)$ is called a Lyapunov pair. In addition, $V$ is called  a Lyapunov function if $W =0$.
\end{definition}
\begin{remark}
(i) Note that if the uniqueness is not available, we can deal with the weak  Lyapunov pairs by using similar arguments, i.e., the inequality (\ref{501}) is satisfied for at least one trajectory of problem (\ref{main}).\\
(ii) Let $x(\cdot):=x(\cdot;t_0,x_0)$. From Theorem \ref{exist}, we know that 
$$
\Vert x(t) \Vert \le \Vert x_0 \Vert+ m(x_0), \;\;\forall \;t\in [t_0,t_0+1],
$$
where $m(x_0)$ is defined in $(\ref{mx0})$.
Combining with Assumptions 1.2 and 2, we have 
\baqn
\Vert (f(t,x(t))-A_{t,x(t)}(x(t)))^0\Vert &\le& c_f(1+\Vert x(t) \Vert)+c_A(1+2\Vert x(t) \Vert)\\
&\le&c_f+c_A+(c_f+2c_A)(\Vert x_0 \Vert+ m(x_0)).
\eaqn
Let
\beq\label{Mx0}
M(x_0):=c_f+c_A+(c_f+2c_A+1)(\Vert x_0 \Vert+ m(x_0)).
\eeq
Then $M(x_0)>m(x_0)$, $(f(t,x(t))-A_{t,x(t)}(x(t))) \cap M(x_0) \ball \neq \emptyset$ for all $t\in [t_0,t_0+1]$ and $M(\cdot)$ is a continuous function w.r.t $x_0$.
\end{remark}

\begin{lemma}\label{zt}
Let $x(\cdot):=x(\cdot;t_0,x_0)$. We define the mapping $Z: [t_0,t_0+1] \rightrightarrows \R^n$ as follows
\beq
Z(t):=(f(t,x(t))-A_{t,x(t)}(x(t))) \cap M(x_0) \ball,
\eeq
where $M(x_0)$ is defined in $(\ref{Mx0})$. Then $Z$ has non-empty, convex compact values with closed graph and uniformly bounded. In particular, $Z$ is upper semi-continuous, i.e.,  given $t_1\in [t_0,T]$, for any $\varepsilon >0$, we can find $\delta>0$ such that
\beq
Z(t)\subset Z(t_1)+\varepsilon \ball,\;\;{\rm for} \; \;\vert t-t_1\vert<\delta.
\eeq
\end{lemma}
\begin{proof}
Obviously, $Z$ has non-empty, convex compact values and uniformly bounded. It remain to check that the graph of $Z$ is closed, i.e., if $y_n\in Z(t_n)$ and $y_n \to y$, $t_n \to t$, we must have $y\in Z(t)$. First we have $y\in  M(x_0) \ball$. Let $B_n:=A_{t_n,x(t_n)}, B=A_{t,x(t)}$ then $B_n, B$ are  maximal monotone operators and 
$$
{\rm dis}(B_n, B)={\rm dis}(A_{t_n,x(t_n)}, A_{t,x(t)})\le L_1\vert t_n-t \vert+L_2\Vert x(t_n)-x(t) \Vert\to 0,\;\;{\rm as}\;n\to +\infty.
$$
On the other hand, we have
$$
f(t_n,x(t_n))-y_n  \in B_n(x(t_n)), \;\; f(t_n,x(t_n))-y_n \to f(t,x(t))-y\;\;\;{\rm and}\;x(t_n)\to x(t).
$$
Using Lemma \ref{closemm}, one obtains that $f(t,x(t))-y\in B(x(t))$, or equivalently $y\in f(t,x(t))-A_{t,x(t)}(x(t))$. Consequently $y\in Z(t)$ and we have the closedness of the graph of $Z$. Then classically, one has the upper semi-continuity of $Z$ (see, e.g., \cite{Clarke}).
\end{proof}
\begin{lemma}\label{existv}
Let $x(\cdot):=x(\cdot;t_0,x_0)$. There exists a sequence a sequence $(t_n)$ such that $t_n\to t_0^+$ and 
\beq
\lim_{n\to +\infty}\frac{x(t_n)-x(t_0)}{t_n-t_0}=v\in Z(t_0),
\eeq
where the mapping $Z$ is defined in Lemma $\ref{zt}$.
\end{lemma}
\begin{proof}
Let 
$$v(t):=  \frac{x(t)-x(t_0)}{t-t_0}, \;t >t_0.
$$
Then $v(t)$ is bounded by $m(x_0)$ (Theorem \ref{exist}) for all $t\in [t_0,t_0+1]$. Hence there exist  a sequence $(t_n)$ and $v\in \R^n$ such that $t_n\to t_0^+$ and $ v_n:=v(t_n)$ converges  to $v$. Let be given $\varepsilon>0$. Note that for $n$ large enough, by using Lemma \ref{zt}, we have 
$$
x'(s)\in Z(s)\subset Z(t_0)+\varepsilon \ball, \;a.e. \; s\in  [t_0, t_n].
$$
Thus
\beq
v_n=\frac{1}{t_n-t_0}\int_{t_0}^{t_n} x'(s)ds \in Z(t_0)+\varepsilon \ball,
\eeq
which implies that $v\in Z(t_0)+\varepsilon \ball$. Since $\varepsilon$ is arbitrary, we deduce that $v\in Z(t_0)$ and the conclusion follows. 
\end{proof}
\noindent  The following result  provides  necessary and sufficient conditions for a Lyapunov pair associated with problem (\ref{main}).
\begin{thm}\label{mainth3}
Let $V \in \Gamma([0,T]\times \R^n)$, $W\in \Gamma_+([0,T]\times \R^n)$, $a\ge 0$ and ${\rm dom}(V)\subset \mathcal{A}_1\;$. Then the following statements are equivalent: \\
$(i)$ For each $(t_0,x_0)\in {\rm dom}(V)$, we have
$$e^{a(t-t_0)}V\big(t,x(t)\big)+\int_{t_0}^tW\big(\tau,x(\tau)\big)d\tau\le V(t_0,x_0)\;\;\forall \;t\ge t_0,$$
where $x(\cdot):=x(\cdot;t_0,x_0)$.\\
$(ii)$ For each $(t_0,x_0)\in {\rm dom}(V)$ and $(\theta,\xi)\in \partial^PV(t_0,x_0)$, we have
\begin{equation}
\theta+\displaystyle\min_{v\in \big(f(t_0,x_0)-A_{t_0,x_0}(x_0) \big)\cap\; M(x_0)  \ball}\big\langle \xi, v \rangle +aV(t_0,x_0)+W(t_0,x_0)\le 0.
\end{equation}
$(iii)$ For each $(t_0,x_0)\in {\rm dom}(V)$,  we have
\begin{equation}
\left\{
\begin{array}{l}
\displaystyle\sup_{(\theta, \xi)\in \partial^PV(t_0,x_0)}\Big\{\theta+\min_{v\in \big(f(t_0,x_0)-A_{t_0,x_0}(x)\big)\;\cap\; M(x_0)\ball}\langle \xi, v \rangle +aV(t_0,x_0)+W(t_0,x_0)\Big\}\le 0,\\ \\
\displaystyle\sup_{(\theta, \xi)\in \partial^\infty V(t_0,x_0)} \Big\{\theta+\min_{v\in \big(f(t,x)-A_{t_0,x_0}(x)\big)\;\cap \; M(x_0)\ball}\langle \xi, v\rangle\Big\}\le 0.
\end{array}\right.
\end{equation}
$(iv)$ For each $(t_0,x_0)\in {\rm dom}(V)$, for any $M>0$ large enough, we have 
\begin{equation}
\left\{
\begin{array}{l}
\displaystyle\sup_{(\theta, \xi)\in \partial^PVt_0,x_0)}\Big\{\theta+\min_{v\in \big(f(t_0,x_0)-A_{t_0,x_0}(x)\big)\;\cap\; M\ball}\langle \xi, v \rangle +aV(t_0,x_0)+W(t_0,x_0)\Big\}\le 0,\\ \\
\displaystyle\sup_{(\theta, \xi)\in \partial^\infty V(t_0,x_0)} \Big\{\theta+\min_{v\in \big(f(t_0,x_0)-A_{t_0,x_0}(x)\big)\;\cap \; M\ball}\langle \xi, v\rangle\Big\}\le 0.
\end{array}\right.
\end{equation}
where $M(x_0)$  is defined in $(\ref{Mx0})$.
\end{thm}
\begin{proof}
Without loss of generality, suppose that  $W$ is Lipschitz continuous on bounded sets (see \cite[Lemma 3.1]{ahat} or \cite{Clarke}). The plan of the proof is  the following: $(i) \Rightarrow (ii)\Rightarrow (iii)\Rightarrow (iv)\Rightarrow (i)$.\\

\noindent $(i) \Rightarrow (ii):$ Let $(t_0,x_0)\in {\rm dom}(V)$ and $(\theta,\xi)\in \partial^PV(t_0,x_0)$. Then  $(\theta,\xi,-1)\in N_{{\rm epi }\;V}^P\big((t_0,x_0),V(t_0,x_0)\big).$ Let $x(\cdot):=x(\cdot;t_0,x_0)$. From $(i)$, one obtains that
$$\Big((t,x(t)), e^{-a(t-t_0)}V(t_0,x_0)-e^{-a(t-t_0)}\int_{t_0}^tW\big(\tau,x(\tau)\big)d\tau\Big)\in {\rm epi}\;V\;\;\forall t\ge t_0.$$
By the definition of $N_{{\rm epi }\;V}^P(y,V(y))$, there exists $\beta>0$ such that for all $t\ge t_0$, one has
\baqn
&&\Big\langle (\theta,\xi,-1), \Big((t,x(t)), e^{-a(t-t_0)}V(t_0,x_0)-e^{-a(t-t_0)}\int_{t_0}^tW\big(\tau,x(\tau)\big)d\tau\Big)-\big((t_0,x_0),V(t_0,x_0)\big)\Big\rangle \\
&\le& \beta \|\Big((t,x(t)), e^{-a(t-t_0)}V(t_0,x_0)-e^{-a(t-t_0)}\int_{t_0}^tW\big(\tau,x(\tau)\big)d\tau\Big)-\big((t_0,x_0),V(t_0,x_0)\big)\Big\|^2,
\eaqn
which is equivalent to 
\beq\label{normal}
\langle   \theta, t-t_0\rangle+ \big\langle \xi, x(t)- x_0\big\rangle-  R(t)\le   \beta (t-t_0)^2+ \beta\|x(t)- x_0\|^2+\beta R^2(t),
\eeq
where 
$$R(t):=(e^{-a(t-t_0)}-1)V(t_0,x_0)-e^{-a(t-t_0)}\int_{t_0}^tW\big(\tau,x(\tau)\big)d\tau.$$
Using Lemma $\ref{existv}$, there exists  a sequence $(t_n)$ such that $t_n\to t_0^+$ and 
$$
\lim_{n\to+\infty} \frac{x(t_n)-x_0}{t_n-t_0}=v
$$
exists and $v\in f(t_0,x_0)-A_{t_0,x_0}(x_0) \big)\cap\; M(x_0)  \ball$.

Taking $t=t_n$ in $(\ref{normal})$. Dividing both sides of $(\ref{normal})$ by $t_n-t_0>0$, letting $n\to +\infty$ , one obtains
$$\theta+\langle \xi, v \rangle+aV(t_0,x_0)+W(t_0,x_0)\le 0.$$
Therefore, we obtain (ii).

\noindent$ (ii) \Rightarrow (iii):$  It remains to check the second inequality of $(iii).$ Let $(\theta,\xi)\in \partial^\infty V(t_0,x_0)$. Then there exist sequences $(\alpha_k)_{k\in \N}\subset \R_+$, $(t_k,x_k)_{k\in\N}$, $(\theta_k,\xi_k)_{k\in \N}$ such that $\alpha_k\to 0^+$, $(t_k,x_k) \rightarrow_V (t_0,x_0)$, $(\theta_k,\xi_k)\in \partial^P V(t_k,x_k)$ and 
$\alpha_k(\theta_k,\xi_k)\to (\theta,\xi).$ For each $k$, one can find $v_k\in \big(f(t_k,x_k)-A_{t_k,x_k}(x_k)\big)\cap M(x_k)\ball$ such that 
\beq\label{appsub}
\theta_k+\langle \xi_k, v_k \rangle + a V(y_k)+ W(y_k)\le 0.
\eeq
Given $\varepsilon>0$, for $k$ large enough, we have $v_k\in \big(-f(t_k,x_k)-A_{t_k,x_k}(x_k)\big)\cap (M(x_0)+\varepsilon)\ball$.
Since the sequence $(v_k)$ is bounded, one can extract a subsequence, without relabelling, and some $v$ such that $v_k\to v$ and $v\in \big(-f(t_0,x_0)-A_{t_0,x_0}(x_0)\big)\cap (M(x_0)+\varepsilon)\ball$. Since $\varepsilon$ is arbitrary,   we have $v\in \big(-f(t_0,x_0)-A_{t_0,x_0}(x_0)\big)\cap M(x_0)\ball$.   Multiplying both sides of (\ref{appsub}) by $\alpha_k$ and let $k\to +\infty$ then one obtains that 
$$
\theta+\langle \xi, v\rangle\le 0,
$$
which implies the  second inequality of $(iii).$\\

\noindent$ (iii) \Rightarrow (iv):$ Obviously for any $M>M(x_0)$.\\

\noindent$ (iv) \Rightarrow (i):$ Let $(t_0, x_0)\in {\rm dom}\;V$.  Let $x(\cdot)=x(\cdot;t_0,x_0)$ be the solution of (\ref{main}) with $x(t_0)=x_0$.
Given any $T>0$, we define the functions $h: [t_0,T]\to \R_+, \gamma:  [t_0,T]\to \R$, $z: [t_0,T]\to [t_0,T]\times \R^n\times \R$ and $\eta: [t_0,T] \to \R_+$ as follows
$$h(t):=\int_{t_0}^tW\big(\tau,x(\tau;x_0)\big)d\tau,\;\; \gamma(t):=e^{-a(t-t_0)}\big(V(t_0,x_0)-h(t)\big),$$
$$z(t):=\big(t,x(t;x_0),\gamma(t)\big),\;\; \eta(t):=\frac{1}{2}d^2\big(z(t),{\rm epi}\;V\big).$$
As in \cite{ahat}, $\eta$ is Lipschitz continuous on every compact interval in $(t_0,T)$ and for all $t\in (t_0,T)$, one has 
$$\partial^C\eta(t)=d\big(z(t),{\rm epi}\;V\big)\partial^Cd\big(z(\cdot),{\rm epi}\;V\big)(t)\neq \emptyset,$$
where $\partial^C$ denotes the Clarke subdifferential. 
 We have then an estimation of $\partial^C\eta$ as in Lemma \ref{estieta}.
Let $t_0<s\le t<T$. By using Gronwall's inequality one has 
\beq
e^{-Nt}\eta(t)\le e^{-Ns}\eta(s),
\eeq
where $N>0$ is defined in Lemma \ref{estieta}.
Let $s\to t_0$ then one has $d(z(t),{\rm epi}\;V)=0$ which implies that
$$e^{a(t-t_0)}V(t,x(t))+\int_{t_0}^tW(\tau, x(\tau))d\tau\le V(t_0,x_0), \;\;\forall \; t_0< t < T.$$
Since $T$ is arbitrary, we obtain the conclusion. 
\end{proof}

\noindent The following lemma  can be used to deduce $ (iv) \Rightarrow (i)$ in the proof of Theorem \ref{mainth3}.
\begin{lemma}\label{estieta}
We can find some $N>0$  such that for almost all $t\in  (t_0,T)$, one has
$$\partial^C\eta(t)\subset (-\infty,N\eta(t)].$$
\end{lemma}
\begin{proof}
Let  $t\in(t_0,T)$ such that $x(\cdot)$ is differentiable at $t$. If $z(t)\in {\rm epi}\;V,$ then $\partial^C\eta(t)=\{0\}$ and the conclusion holds. Otherwise, assume that $z(t)\notin {\rm epi}\;V$. By using \cite[Lemma A.3]{ahat} and noting that  $\dot{x}(t)\in f(t,x(t))-A_{t,x(t)}\big(x(t))$, we obtain 
\beq\label{esz}
\partial^C\eta(t)\subset {\rm \overline{co}}\Bigg[ \bigcup_{(s,u,\mu)\in \mathcal{M}} \langle z(t)-\left( \begin{array}{cc}
s \\ \\
u \\ \\
\mu
\end{array} \right), \left( \begin{array}{cc}
1 \\ \\
\Big(-f(t,x(t))-A_{t,x(t)}\big(x(t)\big)\Big) \\ \\
- a\gamma(t)-e^{-at}W\big(t,x(t)\big)
\end{array} \right)\Bigg],
\eeq
where $\mathcal{M}:={\rm Proj}\big(z(t),{\rm epi}\;V\big)$.
Then it is sufficient to prove that for all $(s, u,\mu)\in \mathcal{M}$ and $\forall x^*\in \Big(f(t,x(t))-A_{t,x(t)}\big(x(t)\big)\Big)$, we have
\beq
\Big \langle z(t)-\left( \begin{array}{cc}
s  \\ \\
u \\ \\
\mu
\end{array} \right), \left( \begin{array}{cc}
1  \\ \\
x^* \\ \\
- a\gamma(t)-e^{-a(t-t_0)}W\big(t,x(t)\big)
\end{array} \right)\Big\rangle\le N\eta(t),
\eeq
for some $N>0.$ Since $(s,u,\mu)\in {\rm Proj}(z(t),{\rm epi}\;V)$, the vector $z(t)-(s,u,\mu)=(t-s,x(t)-u,\gamma(t)-\mu)\in N^P_{{\rm epi}\;V}(s,u,\mu)$. Hence we have  $\gamma(t)-\mu\le 0$. If $\gamma(t)-\mu= 0$ then $(t-s,x(t)-u)\in \partial^\infty V(s,u)$ and if $\gamma(t)-\mu<0$ then $(\frac{t-s}{\mu-\gamma(t)},\frac{x(t)-u}{\mu-\gamma(t)})\in \partial^P V(s,u)$. From $iv)$, there exists $v\in \big(f(s,u)-A_{s,u}(u)\big) $ such that 
\beq\label{eswv}
t-s+\langle x(t)-u, v\rangle\le (\gamma(t)-\mu)\big(aV(s,u)+W(s,u)\big).
\eeq
On the other hand,  we  have  $\sup_{t\in [t_0,T]}\Vert x(t) \Vert \le M_1$ for some $M_1>0$. From the fact that $(s,u,\mu)\in {\rm Proj}(z(t),{\rm epi}\;V)$ and $z(t_0)\in {\rm epi}\;V$, we have 
$$
\Vert z(t)-(s,u,\mu) \Vert \le \Vert z(t)-z(0)\Vert,
$$
which implies that 
$$
\Vert u \Vert \le \Vert (s,u,\mu) \Vert\le 2\Vert z(t) \Vert+ \Vert z(0)\Vert. 
$$
Since $ z(t)$ is uniformly bounded, one can find some $M_2>M_1$ such that $u\in M_2\ball$.
Thanks to  (\ref{eswv}), the hypo-monotonicity of $A$  and the Lipschitz continuous of $f$ on $M_2\ball$, one has  
\baqn
t-s+\big\langle x(t)-u, x^*\big\rangle&=&t-s+\langle x(t)-u,x^*-v+v\rangle\\
&\le& (l_{M_2}+k_{M_2}) \|x(t)-u\|^2+ (\gamma(t)-\mu)\big(aV(s,u)+W(s,u)\big),
\eaqn
where $k_{M_2}$ and $l_{M_2}$ are defined in (\ref{hypo}).
 Note that we already have $\gamma(t)-\mu\le 0$. If $\gamma(t)-\mu < 0$ and suppose that
$V(s,u)\le\gamma(t)$.
One obtains a contradiction
$$
d(z(t),{\rm epi}V)\le d\big(z(t), (s,u,\gamma(t))\big)<d\big(z(t), (s,u,\mu)\big)=d(z(t),{\rm epi}V).
$$
Hence if $\gamma(t)-\mu < 0$, we must have $V(s,u)>\gamma(t)$. Therefore, we always obtain 
$$
\big(\mu-\gamma(t)\big)(\gamma(t)-V(s,u))\le 0.
$$
Consequently,
\baqn
&&t-s+\big\langle x(t)-u, x^*\big\rangle+a\big(\mu-\gamma(t)\big)\gamma(t)+\big(\mu-\gamma(t)\big)e^{-a(t-t_0)}W\big(t,x(t)\big)\\
&&\le  (l_{M_2}+k_{M_2}) \|x(t)-u\|^2+ (\gamma(t)-\mu)\big(aV(s,u)+W(s,u)\big)\\
&&+a\big(\mu-\gamma(t)\big)\gamma(t)+\big(\mu-\gamma(t)\big)e^{-a(t-t_0)}W\big(t,x(t)\big)\\
&&\le  (l_{M_2}+k_{M_2})\|x(t)-u\|^2+a\big(\mu-\gamma(t)\big)(\gamma(t)-V(s,u))+\big(\mu-\gamma(t)\big)\big(W(t,x(t))-W(s,u)\big)\\
&&\le (l_{M_2}+k_{M_2}) \|x(t)-u\|^2+L_W |\mu-\gamma(t)|(|t-s|+ \|x(t)-u\|)\le (\frac{L_W}{2}+ l_{M_2}+k_{M_2})\eta(t),
\eaqn
where $L_W$ is the Lipschitz constant of $W$ on  $[0, T]\times M_2\ball$.
Therefore Lemma $\ref{estieta}$ holds with $N:=\frac{L_W}{2}+ l_{M_2}+k_{M_2}$.  
\end{proof}

Now let us provide some illustrative examples as follows.
\begin{example}
Let be given $p>0$ and a differentiable function $g: [0,+\infty)\to \R$ such that $\dot{g}\le 2g$. We consider the following systems in $\R^3$:
\begin{equation}\label{ex1}
\left\{\begin{array}{l}
\dot{x}_1= -x_1-g(t)x_2,\\ \\
\dot{x}_2=x_1-x_2,\\ \\
\dot{x}_3\in -{\rm Sign}(x_3)+p\vert x_3 \vert.
\end{array}\right.
\end{equation}
Then we can reduce the inclusion (\ref{ex1}) into the our problem (\ref{main}) as follows
$$
\dot{x}\in f(t,x)-A_{t,x}(x),
$$
where
$$
f(t,x)=\left( \begin{array}{cc}
-x_1-g(t)x_2 \\ \\
x_1-x_2 \\ \\
p\vert x_3 \vert
\end{array} \right),
$$
and 
$$
A_{t,y}(x)=\left( \begin{array}{c}
0\\ \\
0 \\ \\
-{\rm Sign}(x_3)
\end{array} \right).
$$
Then all assumptions of Theorem \ref{uniq} are satisfied. Consequently, for each initial condition, the system (\ref{ex1}) has a unique solution. Note that for each $(t,x)\in [0,+\infty)\times \R^{2n}$, we can find some $M>0$ such that $f(t,x)-A_{t,x}(x)\in M\ball$ and hence $(f(t,x)-A_{t,x}(x))\cap M\ball=f(t,x)-A_{t,x}(x)$.
Let 
$$
V(t,x)=\left\{\begin{array}{l}
x_1^2+(1+g(t))x_2^2+\vert x_3\vert, \;{\rm if}\;x_3\le \frac{1}{p},\\ \\
+\infty,\; {\rm if} \;x_3> \frac{1}{p}.
\end{array}\right.
$$
Then $V$ is a lower semi-continuous function and 
$$
\partial^PV(t,x)=\left( \begin{array}{cc}
\dot{g}(t)x_2^2\\ \\
2x_1 \\ \\
2x_2(1+g(t))  \\ \\
\xi(x_3)
\end{array} \right),
$$
where
\begin{equation}
\xi(x_3):=\left\{\begin{array}{ll}
{\rm Sign}(x_3)\;\; {\rm if}\; x_3<\frac{1}{p}, \\ \\
1+\R^+\;\;\;\; {\rm if}\; x_3=\frac{1}{p},\\ \\
\emptyset\;\; \;\;\;\;\;\;\;\;\;\;\;\;{\rm if} \;x_3> \frac{1}{p}.
\end{array}\right.
\end{equation}
 For all $(t,x)\in {\rm dom}(V)$ and $(\theta,\xi)\in \partial^PV(t,x)$ we have
\baqn
&&\theta+\displaystyle\min_{v\in \big(f(t,x)-A_{t,x}(x) \big)}\big\langle \xi, v \rangle\\
&=&\dot{g}(t)x_2^2+2x_1(-x_1-g(t)x_2)+2x_2(1+g(t))(x_1-x_2) +{\rm Sign}(x_3)(-{\rm Sign}(x_3)+p\vert x_3 \vert)\\
&=&-2x_1^2+2x_1x_2-2x_2^2+x_2^2(\dot{g}(t)-2g(t))+{\rm Sign}^2(x_3)(px_3-1)\\
&\le&-x_1^2-x_2^2-(x_1-x_2)^2 \le 0.
\eaqn
Hence, $V$ is a Lyapunov function for (\ref{ex1}) by using Theorem \ref{mainth3}.
\end{example}
\begin{example}
Let be given $\alpha >0, \gamma >0$ and $\beta\in \R$. We consider the following differential inclusion in $\R^2$:
\begin{equation}\label{ex2}
\left\{\begin{array}{l}
\dot{x}_1\in -\alpha x_1+\beta x_2-N_{C(t,x_1)}(x_1),  \\ \\
\dot{x}_2\in -\beta x_1 + x_2- \gamma{\rm Sign}(x_2),\\ \\
x(0)=(x_{01}\; x_{02})^T,
\end{array}\right.
\end{equation}
where $C(t,x_1):=[-(t+2\vert x_{01}\vert ), t+2\vert x_{01}\vert]+ x_1/2$. Then we can rewrite (\ref{ex2}) into our form (\ref{main}) as follows
$$
\dot{x}\in f(t,x)-A_{t,x}(x),
$$
where
$$
f(t,x)=\left( \begin{array}{c}
 -\alpha x_1+\beta x_2 \\ \\
-\beta x_1 +x_2
\end{array} \right),
$$
and 
$$
A_{t,x}(x)=\left( \begin{array}{c}
-N_{C(t,x_1)}(x_1)\\ \\
-\gamma{\rm Sign}(x_2)
\end{array} \right).
$$
Then it is easy to see that $x_{01}\in C(0,x_{01})$ and all assumptions of Theorem \ref{uniq} are satisfied. Let us consider the function $V$ as follows
$$
V(t,x)=\left\{\begin{array}{l}
 \frac{1}{2}x_1^2+  \frac{1}{2}x_2^2, \;{\rm if}\;\vert x_2 \vert\le \gamma,\\ \\
+\infty,\; {\rm if} \;x_2> \gamma.
\end{array}\right.
$$
Then $V$ is a lower semi-continuous function and 
$$
\partial^PV(t,x)=\left( \begin{array}{c}
0 \\ \\
x_1\\ \\
\xi(x_2)
\end{array} \right),
$$
where
\begin{equation}
\xi(x_2):=\left\{\begin{array}{ll}
x_2\;\; {\rm if}\; -\gamma<x_2<\gamma, \\ \\
\{kx_2: k\ge 1\}\;\; {\rm if}\; \vert x_2 \vert=\gamma,\\ \\ 
\emptyset,\; {\rm if} \;\vert x_2 \vert > \gamma.
\end{array}\right.
\end{equation}
  Given $(t,x)\in {\rm dom}(V)$, for any $M$ large enough and for all $(\theta,\xi)\in \partial^PV(t,x)$ one has
\baqn
&&\theta+\displaystyle\min_{v\in \big(f(t,x)-A_{t,x}(x) \big)\cap M\ball}\big\langle \xi, v \rangle\\
&\le& x_1( -\alpha x_1+\beta x_2)+x_2(-\beta x_1 + x_2)-\gamma \vert x_2 \vert \\
&=&-\alpha x_1^2+\vert x_2 \vert (\vert x_2 \vert-\gamma)\\
&\le& 0.
\eaqn
Applying Theorem \ref{mainth3}, we conclude that $V$ is a Lyapunov function for (\ref{ex2}).
\end{example}
\section{Applications for sweeping processes and Lur'e dynamical systems}
In this section, we  show that the obtained results in Sections 3 and 4 can be used to study two well-known  problems: sweeping processes and Lur'e dynamical systems.
\subsection{State-dependent sweeping processes}
Let us consider the case  $A_{t,x}=N_{C(t,x)}$, the normal cone of a moving set, where $C: [0,T]\times \R^n \rightrightarrows \R^n$ has non-empty closed convex values. Thanks to Lemma \ref{hauslem},  Assumptions 1.1 is equivalent to  
\beq\label{stateswp}
d_H(C(t,x), C(s,y))\le L_1 \vert t-s \vert+ L_2 \Vert x-y \Vert,\;\;\forall\; t,s \in [0,T],\; x,y\in \R^n,
\eeq
for some constant $L_1\ge 0,$ $0\le L_2<1$.  Our  problem (\ref{main}) becomes the classical state-dependent sweeping processes. Then Theorem \ref{exist} allows us to obtain the existence result which is accordant with \cite{Kunze}. The case of infinite Hilbert spaces can be done similarly with some compactness assumption. In  \cite{Kunze}, the authors provided some examples to show that the existence result can be lost if $L_2\ge 1.$ So our upper bound for $L_2$ in   problem (\ref{main}) is optimal. In general, we do not have the uniqueness of solutions for  state-dependent sweeping processes (see, e.g.,\cite{Kunze}). Here we give a uniqueness result under the hypo-monotonicity of the normal cone, which is a corollary of Theorem \ref{uniq}. 
\begin{thm}
Let all the assumptions of Theorem \ref{uniq} hold. Then for each $x_0\in C(t_0,x_0)$ with $0\le t_0\le T$, the  following differential inclusion  
\begin{equation}\label{swp}
\left\{\begin{array}{l}
\dot{x}(t)\in f(t,x(t))-N_{C(t,x(t))}(x(t)), \;a.e. \; t\in [t_0,T],  \\ \\
x(t_0)=x_0,
\end{array}\right.
\end{equation}
 has a unique solution on $[t_0,T]$.
\end{thm}

Although the literature for the well-posedness of  sweeping processes is immense, there is still no work studying Lyapunov stability for the state-dependent case. It is why the following corollary of Theorem \ref{mainth3} is interesting. Let us first introduce the admissible set for problem (\ref{swp}):
\beq
\mathcal{A}_2:=\{(t_0,x_0)\in \R_+\times \R^n: x_0 \in C(t_0,x_0)  \}.
\eeq
\begin{thm}
Let $V \in \Gamma_w([0,T]\times \R^n)$, $W\in \Gamma_+([0,T]\times \R^n)$, $a\ge 0$ and ${\rm dom}(V)\subset \mathcal{A}_2\;$. Then the following assertions are equivalent: \\
$(i)$ For each $(t_0,x_0)\in {\rm dom}(V)$, we have
$$e^{a(t-t_0)}V\big(t,x(t)\big)+\int_{t_0}^tW\big(\tau,x(\tau)\big)d\tau\le V(t_0,x_0)\;\;\forall \;t\ge 0,$$
where $x(\cdot):=x(\cdot;t_0,x_0)$.\\

\noindent $(ii)$ For each $(t,x)\in {\rm dom}(V)$ and $(\theta,\xi)\in \partial^PV(t,x)$ we have
\begin{equation}
\theta+\displaystyle\min_{v\in \big(f(t,x)-N_{C(t,x)}(x) \big)\cap\; M(x)  \ball}\big\langle \xi, v \rangle +aV(t,x)+W(t,x)\le 0.
\end{equation}
$(iii)$ For each $(t,x)\in {\rm dom}(V)$  we have
\begin{equation}
\left\{
\begin{array}{l}
\displaystyle\sup_{(\theta, \xi)\in \partial^PV(t,x)}\Big\{\theta+\min_{v\in \big(f(t,x)-N_{C(t,x)}(x)\big)\;\cap\; M(x)\ball}\langle \xi, v \rangle +aV(t,x)+W(t,x)\Big\}\le 0,\\ \\
\displaystyle\sup_{(\theta, \xi)\in \partial^\infty V(t,x)} \Big\{\theta+\min_{v\in \big(f(t,x)-N_{C(t,x)}(x)\big)\cap \; M(x)\ball}\langle \xi, v\rangle\Big\}\le 0.
\end{array}\right.
\end{equation}
$(iv)$ For each $(t,x)\in {\rm dom}(V)$, for any $M>0$ large enough, we have 
\begin{equation}
\left\{
\begin{array}{l}
\displaystyle\sup_{(\theta, \xi)\in \partial^PV(t,x)}\Big\{\theta+\min_{v\in \big(f(t,x)-N_{C(t,x)}(x)\big)\;\cap\; M\ball}\langle \xi, v \rangle +aV(t,x)+W(t,x)\Big\}\le 0,\\ \\
\displaystyle\sup_{(\theta, \xi)\in \partial^\infty V(t,x)} \Big\{\theta+\min_{v\in \big(f(t,x)-N_{C(t,x)}(x)\big)\cap \; M\ball}\langle \xi, v\rangle\Big\}\le 0.
\end{array}\right.
\end{equation}
where $M(x)$ is defined in (\ref{Mx0}).
\end{thm}

\subsection{State-dependent Lur'e dynamical systems}
Now we consider the class of  state-dependent  Lur'e  dynamical systems (\ref{Lure0}) in the Introduction. It is known from (\ref{Lurere}) that 
 we can reduce  (\ref{Lure0})  into the following  first order differential inclusions
\begin{equation}\label{Lure}
\left\{\begin{array}{l}
\dot{x}(t)\in g(t,x(t))-B\Phi(t,x(t),x(t)), \;a.e. \; t\in [0,T],  \\ \\
x(t_0)=x_0,
\end{array}\right.
\end{equation}
where 
\beq
\Phi(t,x,y):=(F^{-1}_{t,y}+D)^{-1}Cx, \;\;(t,x,y)\in [0,T] \times \R^{2n}.
\eeq
%
The admissible set for problem (\ref{Lure}) is defined by 
\beq
\mathcal{A}_3:=\{(t_0,x_0)\in \R_+\times \R^n: (F^{-1}_{t_0,x_0}+D)^{-1}Cx_0 \neq \emptyset \}.
\eeq Let us propose the following assumptions. 
\begin{assumption}
 For every $t\in [0,T]$ and $x\in \R^n$, the operator $F_{t,x}: \R^n \rightrightarrows \R^n $ be a maximal monotone operator and  there exists $L_{F1}\ge0, 0\le L_{F2}\le \frac{c_2}{\Vert C \Vert}$ such that 
\beq
{\rm dis}(F_{t,x}, F_{s,y})\le  L_{F1} \vert t-s \vert+L_{F2}\Vert x-y \Vert,\;\;\forall \;t, s \in  [0,T],
\eeq
where $c_2$ is the smallest positive eigenvalue of $C^TC$.
\end{assumption}

\begin{assumption} The matrix $D$ is positive semidefinite, $C^TC$ is full-rank    and 
 $${\rm ker}(D+D^T)\subset {\rm ker}(PB-C^T)$$
for some symmetric positive definite matrix $P$.
\end{assumption}
\begin{assumption} For all $t\ge 0$, if  $(F^{-1}_{t,y}+D)^{-1}Cx \neq \emptyset$ for some $x, y\in \R^n$, it holds that $ {\rm rge}(D+D^T) \cap (F^{-1}_{t,y}+D)^{-1}Cx\neq \emptyset$.
\end{assumption}
\begin{assumption}  For all $t\in [0,T]$, $x\in \R^n:$ 
 $ {\rm rge}(C) \;\cap {\rm rint(rge}(F^{-1}_{t,x}+D)) \neq \emptyset $.
 \end{assumption}
 \begin{assumption}
The single-valued function $g: [0,T] \times \R^n \to \R^n$ is continuous and there exists $c_f>0$ such that 
$$
\Vert g(t,x) \Vert \le c_f (1+\Vert  x\Vert),\;\;\forall \; (t,x)\in [0,T] \times \R^n .
$$
\end{assumption}

The following lemmas are useful. 

\begin{lemma}\label{fullr}
Suppose that $CC^T$ is full-rank. Then for each $y\in \R^m$, we have
\beq
\Vert C^Ty \Vert \ge \frac{c_2}{\Vert C \Vert } \Vert y \Vert.
\eeq
where $c_2$ is the smallest positive eigenvalue of $C^TC$.
\end{lemma}
\begin{proof}
For all $y\in \R^m$, we have
$$
c_2 \Vert y \Vert^2 \le \langle CC^Ty, y\rangle\le \Vert C \Vert \Vert C^Ty \Vert  \Vert y \Vert,
$$
and the conclusion follows.
\end{proof}
The following result is similar to \cite[Lemma 11]{Leand}, where the case $F_{t,x}\equiv N_{C(t,x)}$ is considered. For the completeness, we recall it here.
\begin{lemma}\label{Lipschitz}
Suppose that Assumptions $3, 4, 5$ are satisfied. Then we can find $\beta_1, \beta_2>0$ such that the single-valued minimal-norm  function 
$\Phi^0:  [0,T] \times \R^{2n}\to  {\rm rge}(D+D^T), (t,x,y)\mapsto \Phi^0(t,x,y)$ satisfies the following properties: \\

a) $\Vert \Phi^0(t,x,y) \Vert \le \beta_1(1+\Vert x \Vert+\Vert y \Vert),\;\;\forall (t,x,y)\in  {\rm dom}(\Phi^0)$.\\

b) $\Vert \Phi^0(t_1,x_1,y_1) -\Phi^0(t_2,x_2,y_2)\Vert^2 \le \beta_2 \Vert x_1-x_2 \Vert^2+ \beta_2 (1+\Vert \Phi^0(t_1,x_1,y_1)\Vert+\Vert \Phi^0(t_2,x_2,y_2)\Vert)(\vert t_1-t_2 \vert+\Vert y_1-y_2  \Vert),\;\; \forall \;(t_i,x_i,y_i)\in  {\rm dom}(\Phi^0), i=1, 2$.
In particular, the function $(t,x)\to \Phi^0(t,x, x)$ is continuous. 
\end{lemma}
\begin{proof}
a) Let be given $(t,x,y)\in {\rm dom}(\Phi^0)$, which deduces that $(F^{-1}_{t,y}+D)^{-1}Cx \neq \emptyset$. Using Assumption $5$, we can find some $z_0\in {\rm rge}(D+D^T) \cap (F^{-1}_{t,y}+D)^{-1}(Cx)={\rm rge}(D+D^T) \cap  \Phi(t,x,y)$. Let us prove that $\Phi^0(t,x,y)=z_0 \in {\rm rge}(D+D^T).$   It is enough to show that $\Vert z_1 \Vert \ge \Vert z_0\Vert$ for all $z_1\in \Phi(t,x,y)$. Note that we can always write uniquely $z_1=z_1^{im}+z_1^{ker}$ where $z_1^{im}\in {\rm rge}(D+D^T), z_1^{ker}\in {\rm ker}(D+D^T)$ satisfy $\langle z_1^{im}, z_1^{ker} \rangle=0$. We have 
\beq\label{belongphi}
z_i\in (F^{-1}_{t,y}+D)^{-1}(Cx) \Leftrightarrow z_i \in F_{t,y}(Cx-Dz_i), \;i=0, 1.
\eeq
Using the monotonicity of $F_{t,y}$ and $D$, we imply  that $\langle D(z_0-z_1), z_0-z_1 \rangle =0$. It means that $z_1-z_0=z_1^{im}+z_1^{ker}-z_0 \in {\rm ker}(D+D^T)$ and hence $z_1^{im}-z_0 \in {\rm ker}(D+D^T) \cap {\rm rge}(D+D^T)=\{0\}$.
Therefore 
\beq\nonumber
\Vert z_1 \Vert^2=  \Vert z_1^{im}\Vert^2+\Vert z_1^{ker} \Vert^2=\Vert z_0  \Vert^2+\Vert z_1^{ker} \Vert^2 \ge \Vert z_0  \Vert^2,
\eeq
and  one obtains that $\Phi^0(t,x,y)=z_0 \in {\rm rge}(D+D^T)$. 

 Fix $(t_0,x_0,x_0)\in {\rm dom}(\Phi^0)$, where $(t_0, x_0)$ is a point in  $\mathcal{A}_3$.    Using the definition of ${\rm dis}(F_{t,y}, F_{t_0,x_0})$ and the fact that 
 $$
 \Phi^0(t_0,x_0,x_0) \in (F^{-1}_{t_0,x_0}+D)^{-1}(Cx_0) \Leftrightarrow \Phi^0(t_0,x_0,x_0) \in F_{t_0,x_0}(Cx_0-D \Phi^0(t_0,x_0,x_0)),
 $$
 $$
  \Phi^0(t,x,y)\in (F^{-1}_{t,y}+D)^{-1}(Cx) \Leftrightarrow  \Phi^0(t,x,y)\in F_{t,y}(Cx-D  \Phi^0(t,x,y)),
 $$
 one has 
 \baq\nonumber
&& \langle C(x-x_0), \Phi^0(t,x,y)-\Phi^0(t_0,x_0,x_0) \rangle \\\nonumber
 &\ge&  \langle D(\Phi^0(t,x,y)-\Phi^0(t_0,x_0,x_0)), \Phi^0(t,x,y)-\Phi^0(t_0,x_0,x_0)\rangle\\\nonumber
& - &  (1+\Vert\Phi^0(t,x,y)\Vert+\Vert \Phi^0(t_0,x_0,x_0)\Vert) (L_{1} \vert t-t_0 \vert+   L_{2}\Vert y-x_0 \Vert)\\\nonumber
 &\ge& c_1  \Vert\Phi^0(t,x,y)-\Phi^0(t_0,x_0,x_0)\Vert^2\\
 &-&(1+\Vert\Phi^0(t,x,y)\Vert+\Vert \Phi^0(t_0,x_0,x_0)\Vert) (L_{1}T+L_{2}\Vert y-x_0 \Vert),
 \label{est1t2}
 \eaq
 where $c_1>0$ is defined in Lemma \ref{estiD}.
  Hence there exists some $\beta>0$ such that 
\baqn
  \Vert\Phi^0(t,x,y)\Vert^2 &\le& \Vert\Phi^0(t,x,y)\Vert (\beta \Vert x \Vert+\beta \Vert y \Vert+ \beta) + \beta (\Vert x \Vert+\beta \Vert y \Vert+1)\\
  &\le& \beta(\Vert\Phi^0(t,x,y)\Vert+1)(\Vert x \Vert+ \Vert y \Vert+1)
\eaqn
 and one obtains the conclusion  with $\beta_1:=2\beta+1$. \\
 
 b) Similarly as in (\ref{est1t2}),  we have for every $(t_i,x_i,y_i)\in  {\rm dom}(\Phi^0), i=1, 2$ that
 \baqn
 && \langle C(x_1-x_2), \Phi^0(t_1,x_1,y_1)-\Phi^0(t_2,x_2,y_2) \rangle 
 \ge c_1  \Vert\Phi^0(t_1,x_1,y_1)-\Phi^0(t_2,x_2,y_2)\Vert^2\\
 &-&(1+\Vert\Phi^0(t_1,x_1,y_1)\Vert+\Vert \Phi^0(t_2,x_2,y_2)\Vert) (L_{1}\vert t_1-t_2\vert+L_{2}\Vert y_1-y_2 \Vert).
 \eaqn
On the other hand
 \baqn
 && \langle C(x_1-x_2), \Phi^0(t_1,x_1,y_1)-\Phi^0(t_2,x_2,y_2) \rangle \\
  &\le& \frac{c_1}{2}\Vert\Phi^0(t_1,x_1,y_1)-\Phi^0(t_2,x_2,y_2)\Vert^2+\frac{\Vert C \Vert^2}{2c_1}\Vert x_1-x_2 \Vert^2,
 \eaqn
 and thus  the conclusion follows. 
\end{proof}
Now we are ready for the well-posedness of (\ref{Lure0}).
\begin{thm} (Existence)\label{exlure}
Let Assumptions 3-7 hold. Then for each $(t_0,x_0)\in \mathcal{A}_3$,  problem (\ref{Lure0}) has a solution. 
\end{thm}
\begin{proof}
The inclusion in (\ref{Lure}) can be rewritten as 
$$
\dot{x}(t)\in g(t,x(t))-(B-C^T)\Phi(t,x(t),x(t))-C^T\Phi(t,x(t),x(t)=f(t,x(t))-A_{t,x(t)}(x(t))
$$
where
\beq\label{fnew}
f(t,x):=g(t,x(t))-(B-C^T)\Phi(t,x(t),x(t))
\eeq
and
\beq\label{Anew}
A_{t,y}(x):=C^T\Phi(t,x,y)=C^T(F^{-1}_{t,y}+D)^{-1}Cx.
\eeq
Then $A$ and $f$ satisfy Assumptions 1.2 and 2 respectively (see Lemma \ref{Lipschitz}). It remains to check that $A$ satisfies Assumption 1.1. Indeed, we have 
\baqn
&&{\rm dis}(A_{t_1,y_1}, A_{t_2,y_2})\\
&=&\sup\Big\{ \frac{\langle C^Tx_1^*-C^Tx_2^*, x_2-x_1 \rangle}{1+ \Vert C^Tx_1^*\Vert+ \Vert C^Tx_2^*\Vert}: x_i^*\in (F^{-1}_{t_i,y_i}+D)^{-1}Cx_i\Big\} \\
&=& \sup\Big\{ \frac{\langle x_1^*-x_2^*, Cx_2-Cx_1 \rangle}{1+ \Vert C^Tx_1^*\Vert+ \Vert C^Tx_2^*\Vert}: x_i^*\in F_{t_i,y_i}(Cx_i-Dx_i^*) \Big\} \\
&\le& \sup\Big\{ \frac{\langle x_1^*-x_2^*, (Cx_2-Dx^*_2)-(Cx_1-Dx^*_1) \rangle}{1+\Vert x_1^*\Vert+\Vert x_2^*\Vert} \frac{1+\Vert x_1^*\Vert+\Vert x_2^*\Vert}{1+ \Vert C^Tx_1^*\Vert+ \Vert C^Tx_2^*\Vert}: x_i^*\in F_{t_i,y_i}(Cx_i-Dx_i^*) \Big\} \\
&&({\rm since\;} D\; {\rm is\;positive\;semidefinite})\\
&\le& \frac{\Vert C \Vert}{c_2} \sup\Big\{ \frac{\langle x_1^*-x_2^*, (Cx_2-Dx^*_2)-(Cx_1-Dx^*_1) \rangle}{1+\Vert x_1^*\Vert+\Vert x_2^*\Vert}: x_i^*\in F_{t_i,y_i}(Cx_i-Dx_i^*) \Big\} \\
&&({\rm using\; Lemma\;  \ref{fullr}})\\
&\le& \frac{\Vert C \Vert}{c_2} {\rm dis}(F_{t_1,y_1}, F_{t_2,y_2})\;({\rm using\; the\; definition \; of}\; {\rm dis}(F_{t_1,y_1}, F_{t_2,y_2}))\\
&\le&  \frac{\Vert C \Vert}{c_2} (L_{1} \vert t_1-t_2 \vert+L_{2}\Vert y_1-y_2 \Vert)\\
&\le& (L'_{1} \vert t_1-t_2 \vert+L'_{2}\Vert y_1-y_2 \Vert),
\eaqn
where $L'_{1} := \frac{\Vert C \Vert L_{1}}{c_2}$ and $L'_{2}:=\frac{\Vert C \Vert L_{2}}{c_2}<1$. Thus, the proof is completed.  
\end{proof}
\begin{remark}
(i) If $F_{t,x}\equiv N_{C(t,x)}$, the full-rankness of $CC^T$ can be relaxed by the condition ${\rm rge}(D)\subset {\rm rge}(C)$ by using nice property of the normal cone \cite{Leand}. \\
(ii) The state-dependent sweeping process is also a special case of $(\ref{Lure0})$ when $F_{t,x}\equiv N_{C(t,x)}$ and $B=C=I, D=0$.
\end{remark}
\begin{thm} (Uniqueness)
Let all assumptions of Theorem \ref{exlure} hold. In addition, suppose that  $F_{t,x}(\cdot)=G_t(\cdot+g_1(t,x) )$ where  $G_t: \R^n \rightrightarrows \R^n$ is a maximal monotone operator for each $t\in  [0,T]$ and $ g_1:   [0,T] \times \R^n \to {\rm rge}(D+D^T)$ and $g$ are  Lipschitz continuous in bounded sets w.r.t the second variable. Then for each $(t_0,x_0)\in \mathcal{A}_3$,  problem (\ref{Lure0}) has a unique solution. 
\end{thm}
\begin{proof}
Using Theorem \ref{uniq}, it is sufficient to show that $A$ defined in (\ref{Anew}) is hypo-monotone in bounded sets. Indeed, given $M>0$, for all $t\in [0,T], x_i\in \R^n \cap M\ball, x^*_i\in A_{t,x_i}(x_i)$, $i=1,2$ we have $y_i=C^Tz^*_i$ where 
$$
z^*_i\in (F^{-1}_{t,x_i}+D)^{-1}Cx_i \Leftrightarrow z^*_i\in F_{t,x_i}(Cx_i-Dz^*_i)=G_t(Cx_i-Dz^*_i+g_1(t,x)).
$$
Thank to the monotonicity of $G_t$, one has
$$
\langle z^*_1-z^*_2, (Cx_1-Dz^*_1+g_1(t,x_1)-(Cx_2-Dz^*_2+g_1(t,x_2)\rangle \ge 0.
$$
Therefore,
\baqn
\langle x^*_1- x^*_2, x_1-x_2 \rangle&=&\langle C^Tz^*_1- C^Tz^*_2, x_1-x_2 \rangle\\
&\ge&\langle D(z^*_1- z^*_2), z^*_1- z^*_2 \rangle-\langle z^*_1-z^*_2,g_1(t,x_1)-g_1(t,x_2)\rangle\\
&\ge& c_1 \Vert z^{*im}_1- z^{*im}_2 \Vert^2 - l_M \Vert z^{*im}_1- z^{*im}_2 \Vert \Vert x_1- x_2 \Vert\\
&\ge& - \frac{l_M^2}{4c_1}\Vert x_1- x_2 \Vert^2,
\eaqn
where $c_1>0$ is defined in Lemma \ref{estiD}, $l_M$ is the Lipschitz constant of $g_1$ on the ball $M\ball$ and $z^{*im}$ denotes the projection of $z^*$ onto ${\rm rge}(D+D^T)$. The proof is therefore completed. 
\end{proof}

We also have a characterization for nonsmooth Lyapunov pairs associated with problem (\ref{Lure0}), which is a consequence of Theorem \ref{mainth3}.
\begin{thm}\label{mainth5}
Let $V \in \Gamma_w([0,+\infty)\times \R^n)$, $W\in \Gamma_+([0,+\infty)\times \R^n)$, $a\ge 0$ and ${\rm dom}(V)\subset \mathcal{A}_3$. Then the following assertions are equivalent: \\
$(i)$ For each $(t_0,x_0)\in {\rm dom}(V)$, we have
$$e^{a(t-t_0)}V\big(t,x(t)\big)+\int_{t_0}^tW\big(\tau,x(\tau)\big)d\tau\le V(t,x)\;\;\forall \;t\ge 0,$$
where $x(\cdot):=x(\cdot;t_0,x_0)$.\\

\noindent $(ii)$ For each $(t,x)\in {\rm dom}(V)$ and $(\theta,\xi)\in \partial^PV(t,x)$ we have
\begin{equation}
\theta+\displaystyle\min_{v\in \big(g(t,x)-B(F^{-1}_{t,x}+D)^{-1}Cx \big)\cap\; M(x)  \ball}\big\langle \xi, v \rangle +aV(t,x)+W(t,x)\le 0.
\end{equation}
$(iii)$ For each $(t,x)\in {\rm dom}(V)$  we have
\begin{equation}\nonumber
\left\{
\begin{array}{l}
\displaystyle\sup_{(\theta, \xi)\in \partial^PV(t,x)}\Big\{\theta+\min_{v\in \big(g(t,x)-B(F^{-1}_{t,x}+D)^{-1}Cx\big)\;\cap\; M(x)\ball}\langle \xi, v \rangle +aV(t,x)+W(t,x)\Big\}\le 0,\\ \\
\displaystyle\sup_{(\theta, \xi)\in \partial^\infty V(t,x)} \Big\{\theta+\min_{v\in \big(g(t,x)-B(F^{-1}_{t,x}+D)^{-1}Cx\big)\cap \; M(x)\ball}\langle \xi, v\rangle\Big\}\le 0.
\end{array}\right.
\end{equation}
$(iv)$ For each $(t,x)\in {\rm dom}(V)$, for any $M>0$ large enough, we have 
\begin{equation}\nonumber
\left\{
\begin{array}{l}
\displaystyle\sup_{(\theta, \xi)\in \partial^PV(t,x)}\Big\{\theta+\min_{v\in \big(g(t,x)-B(F^{-1}_{t,x}+D)^{-1}Cx\big)\;\cap\; M\ball}\langle \xi, v \rangle +aV(t,x)+W(t,x)\Big\}\le 0,\\ \\
\displaystyle\sup_{(\theta, \xi)\in \partial^\infty V(t,x)} \Big\{\theta+\min_{v\in \big(g(t,x)-B(F^{-1}_{t,x}+D)^{-1}Cx\big)\cap \; M \ball}\langle \xi, v\rangle\Big\}\le 0.
\end{array}\right.
\end{equation}
where $M(x)$ is defined in (\ref{Mx0}).
\end{thm}
\begin{proof}
Let us note that for each $(t,x)\in [0,+\infty)\times \R^n$, we have
$$
f(t,x)-A_{t,x}(x)=g(t,x)-B(F^{-1}_{t,x}+D)^{-1}Cx.
$$
Therefore, the conclusion follows by using Theorem  \ref{mainth3}.
\end{proof}

\section{Conclusion} \label{section5}
In the paper, we introduce and study the well-posedness as well as nonsmooth Lyapunov pairs for a class of state-dependent maximal monotone differential inclusions with some illustrative examples. The obtained results can be used to deal with the  Lyapunov stability for state-dependent sweeping processes and Lur'e dynamical systems for the first time. Well-posedness of state-dependent Lur'e systems involving general maximal monotone operators is also considered here.

\
\end{document}